\documentclass[a4paper,12pt,reqno,intlimits]{extarticle}
\usepackage[cp1251]{inputenc}
\usepackage[T2A]{fontenc}
\usepackage[english]{babel}
\usepackage{amsfonts,amsmath,amsxtra,amsthm,amssymb,latexsym,mathrsfs,graphicx,mathtools,cite,nicefrac,url}
\usepackage{algorithm}
\usepackage{algorithmic}
\usepackage{multirow}
\usepackage{authblk}
\usepackage{tabularx}

\theoremstyle{plain}
\newtheorem{theorem}{Theorem}
\newtheorem{lemma}{Lemma}

\newtheorem{definition}{Definition}

\theoremstyle{definition}
\newtheorem{remark}{Remark}

\numberwithin{theorem}{section}
\numberwithin{lemma}{section}
\numberwithin{proposition}{section}
\numberwithin{corollary}{section}
\numberwithin{remark}{section}
\numberwithin{definition}{section}
\numberwithin{example}{section}

\DeclareMathOperator*{\argmin}{arg\,min}

\author[1]{Alexander~A.~Titov}
\author[1,2]{Fedor~S.~Stonyakin}
\author[1]{Mohammad~S.~Alkousa}
\author[2]{Seydamet~S.~Ablaev}
\author[1]{Alexander~V.~Gasnikov}

\affil[1]{Moscow Institute of Physics and Technology, Moscow, Russia,

\text{\small email: a.a.titov@phystech.edu, gasnikov@yandex.ru, mohammad.math84@gmail.com}.}

\affil[2]{V.\,I.\,Vernadsky Crimean Federal University, Simferopol, Russia,

\text{\small email: fedyor@mail.ru}.}

\begin{document}
	
\title{Analogues of Switching Subgradient Schemes for Relatively Lipschitz-Continuous Convex Programming Problems }

\maketitle

\begin{abstract}
Recently some specific classes of non-smooth and non-Lipschitz convex optimization problems were selected by  Yu.~Nesterov along with H.~Lu. We consider convex programming problems with similar smoothness conditions for the objective function and functional constraints. We introduce a new concept of an inexact model and propose some analogues of switching subgradient schemes for convex programming problems for the relatively Lipschitz-continuous objective function and functional constraints. Some class of online convex optimization problems is considered. The proposed methods are optimal in the class of optimization problems with relatively Lipschitz-continuous objective and functional constraints.

\end{abstract}

\section*{Introduction}\label{sec1_introduction}
Different relaxations of the classical smoothness conditions for functions are interesting for a large number of modern applied optimization problems.
In particular, in \cite {Bauschke_2017} there were proposed conditions of the relative smoothness of the objective function, which mean the replacement of the classic Lipschitz condition by the following weak version
\begin{equation}\label{eq01new}
f(y) \leq f(x) + \langle \nabla f(x), y-x \rangle + LV_{d}(y, x),
\end{equation}
for any $x$ and $y$ from the domain of the objective function $f$ and some $L>0$, $V_d(y, x)$  is an analogue of the distance between the points $x$ and $y$ (often called the \textit{Bregman divergence}). Such a distance is widely used in various fields of science, particularly in optimization. Usually, the \emph{Bregman divergence} is defined on the base of the auxiliary 1-strongly convex and continuously-differentiable function $d:Q\subset\mathbb{R}^n\rightarrow\mathbb{R}$ (\textit{distance generating function}) as follows
\begin{equation}\label{3}
V_d(y,x)=d(y) - d(x) - \langle \nabla d(x), y-x \rangle \quad \forall x,y\in Q,
\end{equation}
where $Q$ is a convex closed set, $\langle \cdot,\cdot\rangle$ is a scalar product in $\mathbb{R}^n$. In particular, for the standard Euclidean norm $\|\cdot\|_2$ and the Euclidean distance in $\mathbb{R}^n$, we can assume that $V_d(y,x) = d(y - x) = \frac{1}{2} \|y - x\|_2^2$ for arbitrary $x, y \in Q$. However, in many applications, it often becomes necessary to use non-Euclidean norms. Moreover, the considered condition of relative smoothness in \cite{Bauschke_2017,Main} implies only the convexity (but not strong convexity) of the distance generating function $d$. As shown in \cite{Main}, the concept of relative smoothness makes it possible to apply a variant of the gradient method to some problems which were previously solved only by using interior-point methods. In particular, we talk about the well-known problem of the construction of an optimal ellipsoid which covers a given set of points. This problem is important in the field of statistics and data analysis.

A similar approach to the Lipschitz property and non-smooth problems was proposed in \cite{Lu} (see also \cite{Nesterov_Relative_Smoothness}).  This approach is based on an analogue of the Lipschitz condition for the objective function $f: Q \to \mathbb{R}$ with Lipschitz constant $M_f >0$,  which involves replacing the boundedness of the norm of the subgradient, i.e. $ \| \nabla f(x) \|_* \leq M_f$,  with the so-called {\it relative Lipschitz condition}
$$
    \|\nabla f(x)\|_* \leqslant \frac{M_f \sqrt{2V_d(y,x)}}{\|y-x\|} \quad \forall x, y \in Q, \; y \neq x,
$$
where $\|\cdot\|_*$ denotes the conjugate norm, see Section \ref{sec2_Inexact_Model_for_Relative} below. Moreover, the distance generating function $d$ is not necessarily strongly convex. In \cite{Lu} there were proposed deterministic and stochastic Mirror Descent algorithms for optimization problems with convex relatively Lipschitz-continuous objective functionals. Note that the applications of the relative Lipschitz-continuity to the well-known classical support vector machine (SVM) problem and to the problem of minimizing the maximum of convex quadratic functions (intersection of $n$ ellipsoids problem) were discussed in \cite{Lu}. 

In this paper we propose a new concept of an inexact model for objective functional and functional constraint. More precisely, we introduce some analogues of the concepts of an inexact oracle \cite{Devolder} and an inexact model \cite{Gasnikov} for objective functionals. However, unlike \cite{Devolder,Gasnikov}, we do not generalize the smoothness condition. We relax the Lipschitz condition and consider a recently proposed generalization of {\it relative Lipschitz-continuity} \cite{Lu,Nesterov_Relative_Smoothness}. We propose some optimal Mirror Descent methods, in different settings of Relatively Lipschitz-continuous convex optimization problems.

The Mirror Descent method originated in the works of A.~Nemirovski and D.~Yudin more than 30 years ago \cite{Nemirovski_efficient_1979,Nemirovski_Yudin_complexity} and  was later analyzed in \cite{Beck_2003}. It can be considered as the non-Euclidean extension of subgradient methods. The method was used in many applications \cite{Nazin_2011,Nazin_2014,Nazin_2013}. Standard subgradient methods employ the Euclidean distance function with a suitable step-size in the projection step. The Mirror Descent extends the standard projected subgradient methods by employing a nonlinear distance function with an optimal step-size in the nonlinear projection step \cite{article:luong_weighted_mirror_2016}. The Mirror Descent method not only generalizes the standard subgradient descent method, but also achieves a better convergence rate and it is applicable to optimization problems in Banach spaces, while the subgradient descent is not \cite{article:doan_2019}. Also, in some works \cite{Beck_comirror_2010,Nemirovski_Yudin_complexity} there was proposed an extension of the Mirror Descent method for constrained problems.

Also, in recent years, online convex optimization (OCO) has become a leading online learning framework, due to its powerful modeling capability for a lot of problems from diverse domains. OCO plays a key role in solving problems where statistical information is being updated \cite{article:hazan_beyond_2014,article:hazan_introduction_2016}. There are a lot of examples of such problems: Internet network, consumer data sets or financial market, machine learning applications, such as adaptive routing in networks, dictionary learning, classification and regression (see \cite{Yuan_2018} and references therein). In recent years, methods for solving online optimization problems have been actively developed, in both deterministic and stochastic settings \cite{BLMg,Hao_2017,Jenatton_2016,Orabona_online_2020}.  Among them one can mention the Mirror Descent method for the deterministic setting of the problem \cite{Orabona_online_2015,article:titov_optima_2019} and for the stochastic case \cite{article:mohammad_kim_2019,Gasnikov_stoc_online_2017,Yunwen_online_2020},  which allows to solve problems for an arbitrary distance function.

This paper is devoted to Mirror Descent methods for convex programming problems with a relatively Lipschitz-continuous objective function and functional constraints. It consists of an introduction and 6 main sections. In Section \ref{sec2_Inexact_Model_for_Relative} we consider the problem statement and define the concept of an inexact $(\delta,\phi,V)$--model for the objective function. Also we propose some modifications of the Mirror Descent method for the concept of Model Generality. Section \ref{sec4_relative_Lip} is devoted to some special cases of problems with the properties of relative Lipschitz continuity, here we propose two versions of the Mirror Descent method in order to solve the problems under consideration. In Sections \ref{sec5_stochastic} and \ref{sec6_online} we consider the stochastic and online (OCO) setting of the optimization problem respectively. In Section \ref{sec7_numerical} one can find numerical experiments which demonstrate the efficiency of the proposed methods.

The contribution of the paper can be summarized as follows:
\begin{itemize}
    \item Continuing the development of Yurii Nesterov's ideas in the direction of the relative smoothness and non-smoothness \cite{Nesterov_Relative_Smoothness} there was introduced the concept of an inexact $(\delta, \phi, V)$--model of the objective function. For the proposed model we proposed some variants of the well-known Mirror Descent method, which provides an $(\varepsilon+\delta)$--solution of the optimization problem, where $\varepsilon$ is the controlled accuracy. There was considered the applicability of the proposed method to the case of the stochastic setting of the considered  optimization problem.
    \item We also considered a special case of the relative Lipschitz condition for objective function. The proposed Mirror Descent algorithm was specified for the case of such functions. Furthermore, there was introduced one more modification of the algorithm with another approach to the step selection. There was also considered the possibility of applying the proposed methods to the case of several functional constraints.
    \item We considered an online optimization problem and proposed the modification of the Mirror Descent algorithm for such a case. Moreover, there were conducted some numerical experiments which demonstrate the effectiveness of the proposed methods.

\end{itemize}

\section{Inexact Model for Relative Non-Smooth Functionals and Mirror Descent Algorithm}\label{sec2_Inexact_Model_for_Relative}
Let $(E,\|\cdot\|)$ be a normed finite-dimensional vector space and $E^*$ be the conjugate space of $E$ with the norm:
$$\|y\|_*=\max\limits_x\{\langle y, \, x\rangle, \, \|x\|\leq1\},$$
where $\langle y, \, x\rangle$ is the value of the continuous linear functional $y$ at $x \in E$.

Let $Q\subset E$ be a (simple) closed convex set. Consider two subdifferentiable functions $f, g: Q \to \mathbb{R}$.  In this paper we consider the following optimization problem
\begin{equation} \label{main}
f(x)\rightarrow \min \limits_{\substack{x\in Q, \,g(x)\leq 0 }}.
\end{equation}

Let $d : Q \rightarrow \mathbb{R}$ be any convex (not necessarily strongly-convex) differentiable  function, we will call it the \textit{reference function}. Suppose we have a constant $\Theta_0>0$, such that $d(x^*) \leq \Theta_0^2,$ where $x^*$ is a solution of \eqref{main}. Note that if there is a set, $X_* \subset Q$, of optimal points for the problem \eqref{main},  we may assume that
$$
	\min\limits_{x^* \in X_*} d(x^*) \leq \Theta_0^2.
$$

Let us introduce some generalization of the concept of Relative Lipschitz continuity \cite{Nesterov_Relative_Smoothness}. Consider one more auxiliary function $\phi : \mathbb{R} \rightarrow \mathbb{R}$, which is strictly increasing and $\phi(0) = 0$. Clearly, due to the strict monotonicity of $\phi(\cdot)$, there exists the inverse function $\phi^{-1}(\cdot)$.

\begin{definition}\label{def}
Let $\delta > 0$. We say that $f$ and $g$ admit the $(\delta, \phi, V)$--model at the point $y \in Q$ if
\begin{equation}\label{aux1}
f(x) + \psi_f(y,x) \leq f(y), \quad - \psi_f(y,x) \leq \phi^{-1}_f(V_d(y,x)) +\delta
\end{equation}
\begin{equation}\label{aux2}
g(x) + \psi_g(y,x) \leq g(y), \quad - \psi_g(y,x) \leq \phi^{-1}_g(V_d(y,x)) +\delta,
\end{equation}
where $\psi_f(y, x)$ and $\psi_g(y, x)$ are convex functions on $y$ and $\psi_f(x, x)=\psi_g(x, x) = 0$ for all $x \in Q$.
\end{definition}

Let $h >0$. For problems with a $(\delta, \phi, V)$--model, the proximal mapping operator (Mirror Descent step) is
defined as follows 

$$
    Mirr_h(x,\psi)=\arg\min\limits_{y\in Q}\left\{\psi(y,x)+\frac{1}{h}V_d(y,x)\right\}.
$$

The following lemma describes the main property of this operator.

\begin{lemma}[Main Lemma]\label{main_lemma_1}
Let $f$ be a convex function, which satisfies \eqref{aux1}, $h>0$ and $x^+=h Mirr_h(x,\psi_f)$. Then  for any  $y \in Q$ 
$$
    h(f(x)-f(y)) \leq - h\psi_f(y,x) \leq \phi^*_f(h) + V_d(y,x) - V_d(y,x^+)+h\delta,
$$
where $\phi_f^*$ is the conjugate function of $\phi_f$. 
\end{lemma}
\begin{proof}
From the definition of $x^+$
$$
    x^+=h Mirr_h(x,\psi_f)=\arg\min\limits_{y\in Q}\left\{h\psi_f(y,x)+V_d(y,x)\right\},
$$
for any $y \in Q$, we have
$$
    h\psi_f(y,x)-h\psi_f(x^+,x) + \langle \nabla d(x^+) - \nabla d(x), y-x^+ \rangle  \geq 0.
$$

Further, $h(f(x) - f(y)) \leq  -h\psi_f(y,x) \leq $
\begin{equation*}
    \begin{aligned}
   &\leq -h\psi_f(x^+,x) + \langle \nabla d(x^+) - \nabla d(x), y-x^+ \rangle \\ & 
    =-h\psi_f(x^+,x) + V_d(y,x) - V_d(y,x^+) - V_d(x^+,x)+h\delta \\&
    \leq h \phi^{-1}_f(V_d(x^+,x)) + V_d(y,x) - V_d(y,x^+) - V_d(x^+,x)+h\delta  \\&
    \leq \phi_f^*(h) + \phi_f(\phi_f^{-1}(V_d(x^+,x))) + V_d(y,x) - V_d(y,x^+) - V_d(x^+,x)+h\delta \\&
    = \phi_f^*(h) + V_d(x^+,x) + V_d(y,x) - V_d(y,x^+) - V_d(x^+,x)+h\delta\\&
    =\phi_f^*(h) + V_d(y,x) - V_d(y,x^+)+h \delta.
    \end{aligned}
\end{equation*}

\end{proof}

For problem \eqref{main} with an inexact $(\delta, \phi, V)$--model, we consider a Mirror Descent algorithm, listed as Algorithm \ref{alg1} below. For this proposed algorithm, we will call step $k$ productive if $g(x^k) \leq \varepsilon$, and non-productive if the reverse inequality $g(x^k) > \varepsilon$ holds.  Let $I$ and $|I|$ denote the set of indexes of productive steps and their number, respectively. Similarly, we use the notation $J$ and $|J|$ for non-productive steps. 

\begin{algorithm}
\caption{Modified MDA for $(\delta, \phi, V)$--model.}
\label{alg1}
\begin{algorithmic}[1]
\REQUIRE $\varepsilon>0, \delta>0, \, h^f >0,  h^g>0,  \Theta_0: \,d(x^*)\leq\Theta_0^2.$
\STATE $x^0=\argmin_{x\in Q}\,d(x)$.
\STATE $I=:\emptyset$ and $J=:\emptyset$
\STATE $N\leftarrow0$
\REPEAT
\IF{$g\left(x^{N}\right)\leq \varepsilon +\delta$}

\STATE $x^{N+1}=Mirr_{h^f}\left(x^N,\psi_f\right),$ \quad "productive step"
\STATE  $N \to I$
\ELSE

\STATE $x^{N+1}=Mirr_{h^g}\left(x^N,\psi_g\right),$ \quad  "non-productive step"
\STATE $N \to J$
\ENDIF
\STATE $N\leftarrow N+1$
\UNTIL {$ \Theta_0^2 \leq \varepsilon\left(|J|h^g+|I|h^f\right) - |J|\phi_g^*(h^g)-|I|\phi_f^*(h^f).$}
\ENSURE $ \widehat{x}:= \frac{1}{|I|} \sum\limits_{k \in I} x^k.$
\end{algorithmic}
\end{algorithm}

Let $x^*$ denote the exact solution of the problem \eqref{main}. The next theorem provides the complexity and quality of the proposed Algorithm \ref{alg1}.

\begin{theorem}[Modified MDA for Model Generality]

Let $f$ and $g$ be convex functionals, which satisfy \eqref{aux1}, \eqref{aux2} respectively and $\varepsilon>0, \delta>0$ be fixed positive numbers. Assume that $\Theta_0 >0$ is a known constant such that $d(x^*) \leq \Theta_0^2$. Then, after the stopping of Algorithm \ref{alg1}, the following inequalities hold:
$$ 
    f(\widehat{x}) - f(x^*) \leq  \varepsilon + \delta \quad \text{  and  } \quad  g(\widehat{x}) \leq \varepsilon + \delta.
$$

\end{theorem}

\begin{proof}
By Lemma \ref{main_lemma_1}, we have for all $k \in I$ and $y \in Q$
\begin{equation}\label{eq_prod}
    h^f\left(f(x^k)-f(y)\right) \leq \phi_f^*(h^f) + V_d(y,x^k) - V_d(y,x^{k+1})+h^f\delta,
\end{equation}
Similarly, for all $k \in J$ and $y \in Q$
\begin{equation}\label{eq_non_prod}
    h^g\left(g(x^k)-g(y)\right) \leq \phi_g^*(h^g) + V_d(y,x^k) - V_d(y,x^{k+1})+h^g\delta, 
\end{equation}
Taking summation, in each side of \eqref{eq_prod} and \eqref{eq_non_prod}, over productive and non-productive steps, we get
$$\sum\limits_{k\in I} h^f\left(f(x^k)-f(x^*)\right) + \sum\limits_{k\in J} h^g\left(g(x^k)-g(x^*)\right) \leq$$
$$\leq\sum\limits_{k\in I}\phi_f^*(h^f)+\sum\limits_{k\in J}\phi_g^*(h^g) + \sum\limits_{k}\left(V_d(x^*,x^k) - V_d(x^*,x^{k+1})\right) +\sum_{k\in I}h^f\delta+\sum_{k\in J}h^g\delta \leq $$
$$\sum\limits_{k\in I}\phi_f^*(h^f)+\sum\limits_{k\in J}\phi_g^*(h^g) + \Theta_0^2+\sum_{k\in I}h^f\delta+\sum_{k\in J}h^g\delta.$$

Since for any $k\in J, \; g(x^k)-g(x^*) > \varepsilon +\delta$, we have

$$\sum\limits_{k\in I} h^f\left(f(\widehat{x})-f(x^*)\right) \leq \sum\limits_{k\in I}\phi_f^*(h^f)+\sum\limits_{k\in J}\phi_g^*(h^g)+ \Theta_0^2 - \varepsilon\sum\limits_{k\in J}h^g  +\sum_{k\in I}h^f\delta=$$
$$=|I|\left(\phi_f^*(h^f)+\delta h^f\right)+|J|\phi_g^*(h^g)-|J|h^g\varepsilon  + \Theta_0^2 \leq \varepsilon|I|h^f + \delta|I|h^f.$$

So, for $\widehat{x}:= \frac{1}{|I|} \sum\limits_{k \in I} x^k, $ after the stopping criterion of Algorithm \ref{alg1} is satisfied, the following inequalities hold
$$
f(\widehat{x}) - f(x^*) \leq  \varepsilon + \delta \quad \text{and} \quad  g(\widehat{x}) \leq  \varepsilon + \delta.
$$
\end{proof}

\section{The Case of Relatively Lipschitz-Continuous Functionals}\label{sec4_relative_Lip}
Suppose hereinafter that the objective function $f$ and the constraint $g$ satisfy the so-called relative Lipschitz condition, with constants $M_f>0$ and $M_g>0$, i.e. the functions $\phi_f^{-1}$ and $\phi_g^{-1}$ from \eqref{aux1} and \eqref{aux2} are modified as follows:
\begin{equation}\label{relL1}
\phi^{-1}_f\left(V_d(y,x)\right)=M_f\sqrt{2V_d(y,x)},
\end{equation}
\begin{equation}\label{relL2}
\phi^{-1}_g\left(V_d(y,x)\right)=M_g\sqrt{2V_d(y,x)}
\end{equation}
 Note that the functions $f,g$ must still satisfy the left inequalities in \eqref{aux1},\eqref{aux2}:
 \begin{equation}\label{auxx1}
f(x) + \psi_f(y,x) \leq f(y), \quad - \psi_f(y,x) \leq M_f\sqrt{2V_d(y,x)} +\delta
\end{equation}
\begin{equation}\label{auxx2}
g(x) + \psi_g(y,x) \leq g(y), \quad - \psi_g(y,x) \leq M_g\sqrt{2V_d(y,x)} +\delta,
\end{equation}

 For this particular case we say that $f$ and $g$ admit the $(\delta, M_f, V)$-- and $(\delta, M_g, V)$--model at each point $x \in Q$ respectively. The following Remark 2 provides the explicit form of $\phi_f, \phi_g$ and their conjugate functions  $\phi^*_f, \phi^*_g$.

\begin{remark}
Let $M_f>0$ and $M_g>0$. Then functions $\phi_f$ and $\phi_g$ which correspond to \eqref{relL1} and \eqref{relL2}  are defined as follows:
$$\phi_f(t)=\frac{t^2}{2M_f^2}, \ \ \phi_g(t)=\frac{t^2}{2M_g^2}.$$
Their conjugate functions have the following form:
\begin{equation}\label{conj1}
\phi_f^*(y)=\frac{y^2M_f^2}{2},
\end{equation}
\begin{equation}\label{conj2}
\phi_g^*(y)=\frac{y^2M_g^2}{2}.
\end{equation}
\end{remark}

For the case of relatively Lipschitz-continuous objective function and constraint, we consider a modification  of Algorithm \ref{alg1}, the modified algorithm is listed as Algorithm \ref{alg2}, below. The difference between Algorithms \ref{alg1} and \ref{alg2} is represented in the control of productivity and the stopping criterion.

\begin{algorithm}[htb]
\caption{Mirror Descent for Relatively Lipschitz-continuous functions, version 1}.
\label{alg2}
\begin{algorithmic}[1]
\REQUIRE $\varepsilon>0, \delta>0, M_f>0, M_g >0,\Theta_0: \,d(x^*)\leq\Theta_0^2$
\STATE $x^0=\argmin_{x\in Q}\,d(x).$
\STATE $I=:\emptyset$
\STATE $N\leftarrow0$
\REPEAT
\IF{$g\left(x^N\right)\leq M_g\varepsilon +\delta$}
\STATE $h^f=\frac{\varepsilon}{M_f},$

\STATE $x^{N+1}=Mirr_{h^f}\left(x^N,\psi_f\right),$ \quad "productive step"
\STATE  $N \to I$
\ELSE
\STATE $h^g=\frac{\varepsilon}{M_g},$ 

\STATE $x^{N+1}=Mirr_{h^g}\left(x^N,\psi_g\right),$ \quad "non-productive step"
\ENDIF
\STATE $N\leftarrow N+1$

\UNTIL { $N \geq \frac{2 \Theta_0^2}{\varepsilon^2}$.}
\ENSURE $\widehat{x} := \frac{1}{|I|} \sum\limits_{k \in I} x^k.$
\end{algorithmic}
\end{algorithm}

For the proposed Algorithm \ref{alg2}, we have the following theorem, which provides an estimate of its complexity and the quality of the solution of the problem.

\begin{theorem}
Let $f$ and $g$ be convex functions, which satisfy \eqref{auxx1} and \eqref{auxx2} for $M_f>0$ and $M_g>0$.

Let $\varepsilon >0, \delta>0$ be fixed positive numbers. Assume that $\Theta_0 >0$ is a known constant such that $d(x^*) \leq \Theta_0^2$. Then, after the stopping of Algorithm \ref{alg2}, the following inequalities hold:
$$
    f(\widehat{x}) - f(x^*) \leq M_f \varepsilon + \delta \quad \text{and} \quad g(\widehat{x}) \leq M_g \varepsilon + \delta.
$$

\end{theorem}

\begin{proof}
By Lemma \ref{main_lemma_1}, we have
\begin{equation*}
    \begin{aligned}
    \sum\limits_{k\in I} h^f\left(f(x^k)-f(x^*)\right) + \sum\limits_{k\in J} h^g\left(g(x^k)-g(x^*)\right) \leq \sum\limits_{k\in I}&\phi_f^*(h^f)+\sum\limits_{k\in J}\phi_g^*(h^g) + \\&  +\Theta_0^2+\sum_{k\in I}h^f\delta+\sum_{k\in J}h^g\delta
    \end{aligned}
\end{equation*}

Since for any $k\in J, \; g(x^k)-g(x^*) > M_g\varepsilon +\delta$ we have 

\begin{equation*}
    \begin{aligned}
    \sum\limits_{k\in I} h^f\left(f(\widehat{x})-f(x^*)\right) &\leq \sum\limits_{k\in I}\phi_f^*(h^f)+\sum\limits_{k\in J}\phi_g^*(h^g)+ \Theta_0^2 - M_g\varepsilon\sum\limits_{k\in J}h^g  +\sum_{k\in I}h^f\delta\\&
     = |I|(\phi_f^*(h^f)+\delta h^f)+|J|\phi_g^*(h^g)-|J|\varepsilon^2  + \Theta_0^2.
    \end{aligned}
\end{equation*}

Taking into account the explicit form of the conjugate functions \eqref{conj1}, \eqref{conj2} one can get:
\begin{equation*}
    \begin{aligned}
    h^f\left(f(\widehat{x})-f(x^*)\right) &\leq |I|\left(\frac{M_f^2{h^f}^2}{2}+\delta h^f\right)+|J|\frac{M_g^2 {h^g}^2}{2}-|J|\varepsilon^2  + \Theta_0^2 \\&
    = |I|\left(\frac{\varepsilon^2}{2}+\delta h^f\right)+|J|\frac{\varepsilon^2}{2}-|J|\varepsilon^2  + \Theta_0^2 \\&
    \leq  M_f\varepsilon|I|h^f + \delta|I|h^f,
    \end{aligned}
\end{equation*}

supposing that the stopping criterion is satisfied.

So, for the output value of the form $\widehat{x} = \frac{1}{|I|} \sum\limits_{k \in I} x^k, $ the following inequalities hold:
$$
f(\widehat{x}) - f(x^*) \leq M_f \varepsilon + \delta \quad \text{and} \quad  g(\widehat{x}) \leq M_g \varepsilon + \delta.
$$
\end{proof}

Also, for the case of relatively Lipschitz-continuous objective function and constraint, we consider another modification  of Algorithm \ref{alg1}, which is listed as the following Algorithm \ref{alg2mod}. Note that the difference lies in the choice of steps $h^f, h^g$ and the stopping criterion.

\begin{algorithm}[htb]
\caption{Mirror Descent for Relatively Lipschitz-continuous functions, version 2.}
\label{alg2mod}
\begin{algorithmic}[1]
\REQUIRE $\varepsilon>0, \delta>0, M_f>0, M_g >0, \, \Theta_0: \,d(x^*)\leq\Theta_0^2.$
\STATE $x^0=\argmin_{x\in Q}\,d(x).$
\STATE $I=:\emptyset$ and $J=:\emptyset$
\STATE $N\leftarrow0$
\REPEAT
\IF{$g\left(x^N\right)\leq \varepsilon +\delta$}
\STATE $h^f=\frac{\varepsilon}{M_f^2},$
\STATE $x^{k+1}=Mirr_{h^f}\left(x^N,\psi_f\right),$ \quad  "productive step"
\STATE  $N \to I$
\ELSE
\STATE $h^g=\frac{\varepsilon}{M_g^2},  $
\STATE $x^{N+1}=Mirr_{h^g}\left(x^{N},\psi_g\right),$\quad  "non-productive step"
\STATE $N \to J$
\ENDIF
\STATE $N\leftarrow N+1$
\UNTIL {$\frac{2 \Theta_0^2}{\varepsilon^2} \leq \frac{|I|}{M_f^2} + \frac{|J|}{M_g^2}.$}
\ENSURE $\widehat{x}:= \frac{1}{|I|} \sum\limits_{k \in I} x^k.$
\end{algorithmic}
\end{algorithm}

By analogy with the proof of Theorem 
2 one can obtain the following result concerning the quality of the convergence of the proposed Algorithm \ref{alg2mod}.

\begin{theorem}
Let $f$ and $g$ be convex functions, which satisfy \eqref{auxx1} and \eqref{auxx2} for $M_f>0$ and $M_g>0$. Let $\varepsilon>0, \delta>0$ be fixed positive numbers. Assume that $\Theta_0 >0$ is a known constant such that $d(x^*) \leq \Theta_0^2$.

Then, after the stopping of Algorithm \ref{alg2mod}, the following inequalities hold:
$$
    f(\widehat{x}) - f(x^*) \leq \varepsilon + \delta \quad \text{and} \quad g(\widehat{x}) \leq \varepsilon + \delta.
$$

Moreover, the required number of iterations of Algorithm \ref{alg2mod} does not exceed
$$
   N = \frac{2M^2 \Theta_0^2}{\varepsilon^2},
  \  \text{where} \ M = \max \{M_f, M_g \}.
$$
\end{theorem}

\begin{remark}
Clearly, Algorithms \ref{alg2} and \ref{alg2mod} are optimal in terms of the lower bounds \cite{Nemirovski_Yudin_complexity}. More precisely, let us understand hereinafter the optimality of the Mirror Descent methods as the complexity $O(\frac{1}{\varepsilon^2})$ (it is well-kown that this estimate is optimal for Lipschitz-continuous functionals \cite{Nemirovski_Yudin_complexity}).
\end{remark}

\begin{remark}[The case of several functional constraints]\label{sk}
Let us consider a set of convex functions $f$ and $g_p: Q \rightarrow\mathbb{R}$, $p \in [m] \stackrel{\text{def}}{=} \{1,2, \ldots, m\}$. We will focus on the following constrained optimization problem
\begin{equation}\label{problem_many}
\min\left\{f(x): \;\; x \in Q \;\;\text{and} \;\; g_p(x)\leq 0 \;\; \text{for all} \;\; p \in [m] \right\}.
\end{equation}

It is clear that instead of a set of functionals $\{g_p(\cdot)\}_{p=1}^{m}$ we can consider one functional constraint $g: Q \rightarrow \mathbb{R}$, such that $g(x) = \max_{p \in [m]}\{g_p(x)\}$. Therefore, by this setting, problem \eqref{problem_many} will be equivalent to problem \eqref{main}.

Assume that for any $p \in [m]$, the functional $g_p$ satisfies the following condition 
$$
    -\psi_{g_{p}} (y, x) \leq M_{g_{p}} \sqrt{2 V_{d} (y, x)} + \delta
$$

For problem \eqref{problem_many}, we propose a modification of Algorithms \ref{alg2} and \ref{alg2mod} (the modified algorithms are listed as Algorithm \ref{mod_alg2} and \ref{alg2modmany} in Appendix A). The idea of the proposed modification allows saving the running time of algorithms due to consideration of not all functional constraints on non-productive steps.

\end{remark}

\begin{remark}[Composite Optimization Problems \cite{Beck,bib_Lu,Nesterov}]
Proposed methods are applicable to the composite optimization problems, specifically
$$
 \min\{f(x) + r(x): \quad x\in Q, \; g(x) + \eta(x) \leq 0\},
$$

where $ r, \eta: Q \to \mathbb{R} $ are so-called simple convex functionals (i. e. the proximal mapping operator $Mirr_h(x,\psi)$ is easily computable). For this case, for any $x, y \in Q$, we have
$$ \psi_f (y, x) = \langle \nabla f(x), y-x \rangle + r(y) - r(x) $$
$$ \psi_g (y, x) = \langle \nabla g(x), y-x \rangle + \eta(y) - \eta(x). $$
\end{remark}

\section{Stochastic Mirror Descent Algorithm}\label{sec5_stochastic}
Let us, in this section, consider the stochastic setting of the problem \eqref{main}. This means that we can still use the value of the objective function and functional constraints, but instead of their (sub)gradient, we use their stochastic (sub)gradient. Namely, we consider the first-order unbiased oracle that produces $\nabla f(x,\xi)$ and $\nabla g(x,\zeta)$, where $\xi$ and $\zeta$ are random vectors and $$\mathbb{E}[\nabla f(x,\xi)] = \nabla f(x), \quad \nabla \mathbb{E}[g(x,\zeta)] = \nabla g(x).
$$
Assume that for each $x, y \in Q$
\begin{equation}\label{EquivStoh}
\langle \nabla f(x,\xi), x - y \rangle \leq M_f \sqrt{2V_{d}(y, x)}  \text{ and } \langle \nabla g(x,\zeta), x - y \rangle \leqslant M_g \sqrt{2V_{d}(y, x)},
\end{equation}
where $M_f, M_g > 0$, Let us consider the proximal mapping operator for $f$
$$
    Mirr_h(x,\xi)=\arg\min\limits_{y\in Q}\left\{\frac{1}{h}V_d(y,x)+\langle \nabla f(x,\xi),y \rangle\right\},  
$$
and similarly, we consider the proximal mapping operator for $g$.
The following lemma describes the main property of this operator.

\begin{lemma}\label{lemma_stoc}
Let $f$ be a convex function which satisfies \eqref{aux1}, $h>0, \delta>0,$ $\xi$ be a random vector and $\Tilde{x} = Mirr_h(x,\xi)$. Then for all $y \in Q$
$$
    h(f(x)-f(y)) \leq  \phi_f^*(h) + V_d(y,x) - V_d(y,\Tilde{x}) +h\langle \nabla f(x,\xi)-\nabla f(x),y-x \rangle+ h\delta,
$$
where, as earlier, $\phi^*_f(h)= \frac{h^2M_f^2}{2}.$
\end{lemma}

Suppose  $\varepsilon > 0$ is a given positive real number. We say that a (random) point $\widehat{x} \in Q$ is an expected $\varepsilon$--solution to the problem \eqref{main}, in the stochastic setting, if
\begin{equation}\label{expected_sol}
\mathbb{E}[f(\widehat{x})] - f(x^*) \leq \varepsilon \;\; \text{and}\;\; g(\widehat{x}) \leq \varepsilon. 
\end{equation}

In order to solve the stochastic setting of the considered problem \eqref{main}, we propose the following algorithm.

\begin{algorithm}[htb]
\caption{Modified Mirror Descent for the stochastic setting.}
\label{alg5}
\begin{algorithmic}[1]
\REQUIRE $\varepsilon>0, \delta >0, \, h^f>0, h^g>0, \Theta_0: \,d(x^*)\leq\Theta_0^2.$
\STATE $x^0=\argmin_{x\in Q}\,d(x).$
\STATE $I=:\emptyset$ and $J=:\emptyset$
\STATE $N\leftarrow0$
\REPEAT
\IF{$g\left(x^{N}\right)\leq \varepsilon +\delta$}

\STATE $x^{N+1}=Mirr_{h^f}\left(x^{N},\xi^{N},\psi_f\right),$ \quad  "productive step"
\STATE  $N \to I$
\ELSE

\STATE $x^{N+1}=Mirr_{h^g}\left(x^{N}, \zeta^N,\psi_g\right),$ \quad  "non-productive step"
\STATE  $N \to J$
\ENDIF
\STATE $N\leftarrow N+1$
\UNTIL {$ \Theta_0^2 \leq \varepsilon\left(|J|h^g+|I|h^f\right) - |J|\phi_g^*(h^g)-|I|\phi_f^*(h^f).$}
\ENSURE $\widehat{x} := \frac{1}{|I|} \sum\limits_{k \in I} x^k.$
\end{algorithmic}
\end{algorithm}

The following theorem gives information about the efficiency of the algorithm. The proof of this theorem is given in Appendix B.
\begin{theorem}\label{theorem_stoc_setting}
Let $f$ and $g$ be convex functions, which satisfy \eqref{aux1} and \eqref{aux2}. Let $\varepsilon>0, \delta>0$ be fixed positive numbers. Then, after the stopping of Algorithm \ref{alg5}, the following inequalities hold:

$$\mathbb{E} [f(\widehat{x} )]-f(x^*)\leq \varepsilon + \delta \quad \text{and} \quad g(\widehat{x})\leq \varepsilon + \delta.$$
\end{theorem}

\begin{remark}
It should be noted how the optimality of the proposed method can be understood. With the special assumptions \eqref{auxx1} -- \eqref{auxx2} and choice of $h^f,h^g$, the complexity of the algorithm is $O(\frac{1}{\varepsilon^2}),$ which is optimal in such class of problems.
\end{remark}

\section{Online Optimization Problem}\label{sec6_online}
In this section we consider the online setting of the  optimization problem \eqref{main}. Namely
\begin{equation} \label{online}
\frac{1}{N}\sum\limits_{i=1}^{N} f_i(x)\rightarrow \min \limits_{\substack{x\in Q, \, g(x)\leq 0}},
\end{equation}
under the assumption that all $f_i: Q \to \mathbb{R}$ ($i=1, \ldots, N$) and g satisfy \eqref{auxx1} and \eqref{auxx2} with constants $M_i >0, i=1, \ldots, N$ and $M_g>0$.

In order to solve problem \eqref{online}, we propose an algorithm  (listed as Algorithm \ref{alg6} below). This algorithm produces $N$ productive steps and in each step, the (sub)gradient of exactly one functional of the objectives is calculated. As a result of this algorithm, we get a sequence $\{x^k\}_{k \in I}$ (on productive steps), which can be considered as a solution to problem \eqref{online}
with accuracy  $\kappa$ (see \eqref{accuracy_online}).

Assume that $M=\max\{M_i,M_g\}, h^f=h^g=h=\frac{\varepsilon}{M}$.

\begin{algorithm}[htb]
\caption{Modified Mirror Descent for the online setting.}
\label{alg6}
\begin{algorithmic}[1]
\REQUIRE $\varepsilon>0, \delta>0, M >0,  N, \Theta_0: \,d(x^*)\leq \Theta_0^2. $
\STATE $x^0=\argmin_{x\in Q}\,d(x).$
\STATE $i:=1, k:=0$
\STATE set $h = \frac{\varepsilon}{M^2}$
\REPEAT
\IF{$g\left(x^k\right)\leq \varepsilon+\delta$}

\STATE $x^{k+1}=Mirr_{h}\left(x^k,\psi_{f_i}\right),$ \quad "productive step"
\STATE $i=i+1,$
\STATE $k=k+1,$
\ELSE

\STATE $x^{k+1}=Mirr_{h}\left(x^k,\psi_g\right),$ \quad "non-productive step"
\STATE $k=k+1,$
\ENDIF
\UNTIL {$i=N+1.$}
\STATE Guaranteed accuracy: 
\begin{equation}\label{accuracy_online}
    \kappa =  \frac{|J|}{N}\left(-\frac{\varepsilon}{2}\right) +\left(\frac{\varepsilon}{2}+\delta\right)+
\frac{M^2\Theta_0^2}{N\varepsilon}.
\end{equation}
\end{algorithmic}
\end{algorithm}

For Algorithm \ref{alg6}, we have the following result.
\begin{theorem}\label{theorem_online}
Suppose all $f_i: Q \to \mathbb{R}$ ($i=1, \ldots, N$) and g satisfy \eqref{auxx1} and \eqref{auxx2} with constants $M_i>0, i=1, \ldots, N$ and $M_g>0$, Algorithm \ref{alg6} works exactly $N$ productive steps. Then after the stopping of this Algorithm, the following inequality holds
$$\frac{1}{N}\sum\limits_{i=1}^{N}f_i(x^k)-\min\limits_{x \in Q}\frac{1}{N}\sum\limits_{i=1}^{N}f_i(x) \leq \kappa, $$
moreover, when the regret is non-negative, there will be no more than $O(N)$ non-productive steps.
\end{theorem}

The proof of this theorem is given in Appendix C. In particular, note that the proposed method is optimal \cite{article:hazan_beyond_2014}: if for some $C>0$, $\kappa\sim \varepsilon\sim \delta = \frac{C}{\sqrt{N}}$, then $|J|\sim O(N)$.

\section{Numerical Experiments}\label{sec7_numerical}
To show the practical performance of the proposed Algorithms \ref{alg2}, \ref{alg2mod} and their modified versions, Algorithm \ref{mod_alg2} and \ref{alg2modmany}, in the case of many functional constraints, a series of numerical experiments were performed\footnote{All experiments were implemented in Python 3.4, on a computer fitted with Intel(R) Core(TM) i7-8550U CPU @ 1.80GHz, 1992 Mhz, 4 Core(s), 8 Logical Processor(s). RAM of the computer is 8 GB.}, for the well-known \textit{Fermat-Torricelli-Steiner} problem, but with some non-smooth functional constraints. 

For a given  set  $\left\{P_k=(p_{1k},p_{2k},\ldots,p_{nk}); \, k\in [r] \right\}$ of $r$ points, in $n$-dimensional Euclidean space $\mathbb{R}^n$, we need to solve the considered optimization problem \eqref{main}, where the objective function $f$ is given by 
	\begin{equation}\label{obj_Fermat}
		f(x):=\frac{1}{r}\sum\limits_{k=1}^r\sqrt{(x_1-p_{1k})^2+\ldots+(x_n-p_{nk})^2} = \frac{1}{r}\sum\limits_{k=1}^r \left\|x - P_k\right \|_2. 
	\end{equation}
	
The functional constraint has the following form
\begin{equation}\label{functional_constraints}
g(x) = \max\limits_{  i \in [m]} \{ g_i(x) = \alpha_{i1} x_1 + \alpha_{i2} x_2 + \ldots +  \alpha_{in}x_n  \}.
\end{equation}
The coefficients $\alpha_{i1}, \alpha_{i2}, \ldots, \alpha_{in}$, for all $i \in [m]$, in \eqref{functional_constraints} and the coordinates of the points $P_k$, for all $k \in [r]$,  are drawn from the normal (Gaussian) distribution with the location of the mode equaling 1 and the scale parameter equaling 2. 

We choose the standard Euclidean norm and the Euclidean distance function in $\mathbb{R}^n$, $\delta = 0$, starting point $x^{0} = \left( \frac{1}{\sqrt{n}}, \ldots , \frac{1}{\sqrt{n}} \right) \in \mathbb{R}^n$ and $Q$ is the unit ball in $\mathbb{R}^n$.

We run Algorithms  \ref{alg2}, \ref{alg2mod}, \ref{mod_alg2} and \ref{alg2modmany}, for $m = 200, n = 500, r = 100$ and different values of $\varepsilon \in \{\frac{1}{2^i}: i=1,2,3,4,5\}$. The results of the work of these algorithms are represented in Table \ref{results_Fermat} below. These results demonstrate the comparison of the number of iterations (Iter.), the running time (in seconds) of each algorithm and the qualities of the solution, produced by these algorithms with respect to the objective function $f$ and the functional constraint $g$, where we calculate the values of these functions at the output $x^{\text{out}} := \widehat{x}$ of the algorithms. We set $f^{\text{best}} := f\left(x^{\text{out}}\right)$ and $g^{\text{out}} := g\left(x^{\text{out}}\right)$.

\begin{table}[htb]
	\centering
	\caption{The results of Algorithms \ref{alg2}, \ref{alg2mod}, \ref{mod_alg2} and \ref{alg2modmany}, with $m = 200, n = 500, r = 100$ and different values of $\varepsilon$.}
	\label{results_Fermat}

\begin{tabular}{|c||c|p{1.4cm}|c|c||c|p{1.4cm}|c|c|}
		\hline  %
		\multirow{2}{*}{} & \multicolumn{4}{c||}{Algorithm \ref{alg2} } & \multicolumn{4}{c|}{ Algorithm \ref{mod_alg2} } \\ \cline{2-9} 
	$1/\varepsilon$	& Iter. & Time (sec.) & $f^{\text{best}}$ & $g^{\text{out}}$& Iter. &  Time (sec.)& $f^{\text{best}}$ & $g^{\text{out}}$ \\ \hline
		$2$	&16 &5.138 &22.327427  &2.210041 &16 &4.883 &22.327427 &2.210041  \\ 
		$4$	&64 &20.911 &22.303430  &2.016617 &64 &20.380 &22.303430 &2.016617  \\ 
		$8$ &256 &84.343 &22.283362  &1.858965 &256 &79.907 &22.283362 & 2.015076  \\ 
		$16$&1024 &317.991 &22.274366  &1.199792&1024 &317.033 &22.273177 & 1.988190\\
		$32$&4096 &1253.717 &22.272859  &0.607871 &4096 &1145.033 &22.269038 & 1.858965 \\ \hline \hline
		
		& \multicolumn{4}{c||}{Algorithm \ref{alg2mod} } & \multicolumn{4}{c|}{ Algorithm \ref{alg2modmany}  } \\ \hline
		$2$	&167 &9.455 & 22.325994 &0.417002 &164 &7.373 &22.325604 &0.391461  \\ 
		$4$	&710 &39.797 &22.305980  &0.204158 &667 &29.954 &22.305654 & 0.188497 \\ 
		$8$ &2910 &158.763 &22.289320  &0.103493 &2583 &119.055 &22.289302 &0.088221  \\ 
		$16$&11613 &626.894 &22.280893  &0.051662 &10155 &468.649 &22.280909 &0.045343 \\
		$32$&46380 &2511.261 &22.277439  &0.026000 &40149 &1723.136 &22.277450 & 0.022639\\ \hline
	\end{tabular}
\end{table}

In general, from the conducted experiments, we can see that Algorithm \ref{alg2} and its modified version (Algorithm \ref{mod_alg2}) work faster than Algorithms \ref{alg2mod} and its modified version (Algorithm \ref{alg2modmany}). But note that Algorithms \ref{alg2mod} and \ref{alg2modmany} guarantee a better quality of the resulting solution to the considered problem, with respect to the objective function $f$ and the functional constraint \eqref{functional_constraints}. Also,  we can see the  efficiency of the modified Algorithm \ref{alg2modmany}, which saves the running time of the algorithm, due to consideration of not all functional constraints on non-productive steps.

\section*{Conclusion}

In the paper, there was introduced the concept of an inexact $(\delta, \phi, V)$--model of the objective function. There were considered some modifications of the Mirror Descent algorithm, in particular for stochastic and online optimization problems. A significant part of the work was devoted to the research of a special case of relative Lipschitz condition for objective function and functional constraints. The proposed methods are applicable for a wide class of problems because relative Lipschitz-continuity is an essential generalization of the classical Lipschitz-continuity. However, for relatively Lipschitz-continuous problems, we could not propose adaptive methods like \cite{Bay}. Note that Algorithm \ref{alg2mod} and its modified version Algorithm \ref{alg2modmany} are partially adaptive since the resulting number of iterations is not fixed, due to the stopping criterion, although the step-sizes are fixed.\\
The authors are very grateful to Yurii Nesterov for fruitful discussions

\newpage
\section*{Appendix A. Modified Algorithms for Problems with Several Functional Constraints}

\begin{algorithm}[htb]
\caption{Modified MDA for Relatively Lipschitz-continuous functions, version 2, several functional constraints. (The modification of Algorithm \ref{alg2})}
\label{mod_alg2}
\begin{algorithmic}[1]
\REQUIRE $\varepsilon>0, \delta>0, M_f>0, M_g >0,\Theta_0: \,d(x^*)\leq\Theta_0^2$.
\STATE $x^0=\argmin_{x\in Q}\,d(x).$
\STATE $I=:\emptyset$
\STATE $N\leftarrow0$
\REPEAT
\IF{$g(x^N)\leq M_g\varepsilon +\delta$}
\STATE $h^f=\frac{\varepsilon}{M_f},$
\STATE $x^{N+1}=Mirr_{h^f}(x^N,\psi_f),$ \quad \quad \quad \quad "productive step"
\STATE  $N \to I$
\ELSE  
\STATE { // $g_{p(N)}(x^N) > M_g\varepsilon +\delta$ for some $p(N) \in [m]$}

\STATE {$h^{g_{p(N)}}= \frac{\varepsilon}{M_{g_{p(N)}}}$ \quad // $M_{g_{p(N)}}$ is the Lipschitz constant of the constraint $g_{p(N)}$.}
\STATE $x^{N+1}=Mirr_{h^{g_{p(N)}}}(x^N,\psi_{g_{p(N)}}),$ \quad  "non-productive step"
\ENDIF
\STATE $N\leftarrow N+1$
\UNTIL { $N \geq \frac{2 \Theta_0^2}{\varepsilon^2}$.}
\ENSURE $ \widehat{x} := \frac{1}{|I|} \sum\limits_{k \in I} x^k.$
\end{algorithmic}
\end{algorithm}

For the proposed modified Algorithm \ref{mod_alg2}, the following result holds.

\begin{theorem}
Let $f$ and $g$ be convex functions, which satisfy \eqref{auxx1} and \eqref{auxx2} for $M_f>0$ and $M_g>0$. Let $\varepsilon>0, \delta>0$ be fixed positive numbers. Assume that $\Theta_0 >0$ is a known constant such that $d(x^*) \leq \Theta_0^2$.
Then, after the stopping of Algorithm \ref{mod_alg2}, the following inequalities hold
$$
    f(\widehat{x}) - f(x^{\ast}) \leq M_g\varepsilon + \delta \quad \text{and} \quad g_{p(k)}(\widehat{x}) \leq M_g\varepsilon + \delta,
$$
where, by $g_{p(k)}$ we mean any constraint such that the inequality $g_{p(k)}(x^k)> M_g \varepsilon + \delta$ holds.
\end{theorem}

\begin{algorithm}[htb]
\caption{Modified MDA for Relatively Lipschitz-continuous functions, version 2, several functional constraints. (The modification of Algorithm \ref{alg2mod})}
\label{alg2modmany}
\begin{algorithmic}[1]
\REQUIRE $\varepsilon>0, \delta>0, M_f>0, \, \Theta_0: \,d(x^*)\leq\Theta_0^2.$
\STATE $x^0=\argmin_{x\in Q}\,d(x).$
\STATE $I=:\emptyset$  and $J=:\emptyset$
\STATE $N\leftarrow0$
\REPEAT
\IF{$g(x^{N})\leq \varepsilon +\delta$}
\STATE $h^f=\frac{\varepsilon}{M_f^2}, $
\STATE $x^{N+1}=Mirr_{h^f}(x^N,\psi_f),$ \quad \quad \quad \quad "productive step"
\STATE  $N \to I$
\ELSE 
\STATE // $g_{p(N)}(x^N) > \varepsilon +\delta$ for some $p(N) \in [m]$
\STATE $h^{g_{p(N)}}= \frac{\varepsilon}{M_{g_{p(N)}}^2}$ \quad // $M_{g_{p(N)}}$ is the Lipschitz constant of the constraint $g_{p(N)}$. 
\STATE $x^{N+1}=Mirr_{h^{g_{p(N)}}}(x^{N},\psi_{g_{p(N)}}),$ \quad  "non-productive step"
\STATE  $N \to J$
\ENDIF
\STATE $N\leftarrow N+1$
\UNTIL {$\frac{2 \Theta_0^2}{\varepsilon^2} \leq \frac{|I|}{M_f^2} + \sum\limits_{k \in J}\frac{1}{M_{g_{p(k)}}^2}.$}
\ENSURE $\widehat{x}:=\frac{1}{|I|} \sum\limits_{k \in I} x^k.$
\end{algorithmic}
\end{algorithm}

Similarly, for the proposed modified Algorithm \ref{alg2modmany}, we have the following result.
 
\begin{theorem}
Let $f$ and $g$ be convex functions, which satisfy \eqref{auxx1} and \eqref{auxx2} for $M_f>0$ and $M_g>0$. Let $\varepsilon>0, \delta>0$ be fixed positive numbers. Assume that $\Theta_0 >0$ is a known constant such that $d(x^*) \leq \Theta_0^2$.

Then, after the stopping of Algorithm \ref{alg2modmany}, the following inequalities hold
$$
    f(\widehat{x}) - f(x^{\ast}) \leq \varepsilon + \delta \quad \text{and} \quad g_{p(k)} (\widehat{x}) \leq \varepsilon + \delta.
$$
Moreover, if $g(x) = \max\limits_{ p \in [m] } \{g_p (x)\}$ satisfies \eqref{conj2}, then the required number of iterations of Algorithm \ref{alg2modmany} does not exceed
$$N = \frac{2M^2 \Theta_0^2}{\varepsilon^2}, \  \text{where} \ M = \max \{M_f, M_g \}.$$
\end{theorem}

\section*{Appendix B. The proof of Theorem \ref{theorem_stoc_setting}.}

Denote
\begin{equation}
\gamma_k=\begin{cases}
   \langle \nabla f(x^k,\xi^k)-\nabla f(x^k),x^*-x^k \rangle  &\text{if $k\in I$,}\\
   \langle \nabla g(x^k,\zeta^k)-\nabla g(x),x^*-x^k \rangle  &\text{if $ k\in J$.}
 \end{cases}
\end{equation}

By Lemma \ref{lemma_stoc}, we have for all $k \in I$
\begin{equation*}
    \begin{aligned}
    h^f\left(f(x^k)-f(x^*)\right) \leq \phi_f^*(h) + V_d(x^k,x^*) - V_d(&x^{k+1},x^*)+ \\& + h^f \left\langle \nabla f(x^k,\xi^k)-\nabla f(x^k),x^*-x^k \right\rangle+ h^f\delta,
    \end{aligned}
\end{equation*}

the same for all $k \in J$, we have
\begin{equation*}
    \begin{aligned}
    h^g\left(g(x^k)-g(x^*)\right) \leq \phi_g^*(h) + V_d(x^k,x^*) - V_d(&x^{k+1},x^*)+\\& +h^g \left\langle \nabla g(x^k,\zeta^k)-\nabla g(x^k),x^*-x^k \right\rangle+ h^g\delta.
    \end{aligned}
\end{equation*}

By taking summation, in each side of both previous inequalities, over productive and non-productive
steps, we get

$$
    \sum\limits_{k\in I} h^f\left(f(x^k)-f(x^*)\right) + \sum\limits_{k\in J} h^g\left(g(x^k)-g(x^*)\right) \leq
$$

$$
    \leq\sum\limits_{k\in I}\phi_f^*(h^f)+\sum\limits_{k\in J}\phi_g^*(h^g) + \sum\limits_{k}\left(V_d(x^*,x^k) - V_d(x^*,x^{k+1})\right) +\sum_{k\in I}(h^f\delta+\gamma_k)+\sum_{k\in J}(h^g\delta +\gamma_k)\leq
$$

$$
    \sum\limits_{k\in I}\phi_f^*(h^f)+\sum\limits_{k\in J}\phi_g^*(h^g) + \Theta_0^2+\sum_{k\in I}h^f\delta+\sum_{k\in J}h^g\delta+\sum\limits_{k\in I}h^f\gamma_k +\sum\limits_{k\in J}h^g\gamma_k.
$$

For each $k \in J, g(x^k)-g(x^*) > \varepsilon +\delta$ and we have

\begin{equation*}
    \begin{aligned}
    \sum\limits_{k\in I} h^f(f(\widehat{x})-f(x^*)) & \leq \sum\limits_{k\in I}\phi_f^*(h^f)+\sum\limits_{k\in J}\phi_g^*(h^g)+ \Theta_0^2 - \varepsilon\sum\limits_{k\in J}h^g  +\sum_{k\in I}h^f\delta \\&
    + \sum\limits_{k\in I}h^f\gamma_k +\sum\limits_{k\in J}h^g\gamma_k = |I|\left(\phi_f^*(h^f)+\delta h^f\right)+|J|\phi_g^*(h^g)-|J|\varepsilon h^g   \\&
    + \Theta_0^2+\ \sum\limits_{k\in I}h^f\gamma_k +\sum\limits_{k\in J}h^g\gamma_k  \leq \varepsilon|I|h^f+|I|h^f\delta+ \sum\limits_{k\in I}h^f\gamma_k +\sum\limits_{k\in J}h^g\gamma_k.
    \end{aligned}
\end{equation*}

Now from the definition of $\widehat{x}$ (the Ensure of Algorithm \ref{alg5}) and by taking the expectation we obtain
$$\mathbb{E} [f(\widehat{x} )]-f(x^*)\leq \varepsilon +\delta+ \mathbb{E}\left[\sum\limits_{k\in I}\frac{\gamma_k}{|I|}\right] + \mathbb{E}\left[\sum\limits_{k\in J}\frac{\gamma_k}{|J|}\right] = \varepsilon+\delta,$$
as well as  $g(\widehat{x})\leq \varepsilon+\delta.$

\section*{Appendix C. The proof of Theorem \ref{theorem_online}.}

By Lemma \ref{main_lemma_1}, we have for all $k \in I$
\begin{equation}\label{eq_online_productive}
    h\left(f_i(x^k)-f_i(y)\right) \leq \phi^*(h) + V_d(y,x^k) - V_d(y,x^{k+1})+h\delta, 
\end{equation}
the same for all $k \in J$, we have
\begin{equation}\label{eq_online_non_productive}
    h\left(g(x^k)-g(y)\right) \leq \phi^*(h) + V_d(y,x^k) - V_d(y,x^{k+1})+h\delta. 
\end{equation}
By taking summation, in each side of \eqref{eq_online_productive} and \eqref{eq_online_non_productive}, over productive and non-productive
steps, we get
\begin{equation*}
    \begin{aligned}
    \sum\limits_{i = 1}^{N} h\left(f_i(x^k)-f_i(x^*)\right) + \sum\limits_{k\in J} h\left(g(x^k)-g(x^*)\right) &\leq \left(N+|J|\right)(\phi^*(h)+h\delta) + \\&  + \sum\limits_{k}\left(V_d(x^*,x^k) - V_d(x^*,x^{k+1})\right)\\&
    \leq\left(N+|J|\right)(\phi^*(h)+h\delta) + \Theta_0^2.
    \end{aligned}
\end{equation*}

Then
\begin{equation*}
    \begin{aligned}
    \sum\limits_{i=1}^{N}f_i(x^k)-f_i(x^*) &\leq
|J|\left(\frac{M^2\phi^*(h)}{\varepsilon}+\delta\right) +N\left(\frac{M^2\phi^*(h)}{\varepsilon}+\delta\right)+
\frac{M^2\Theta_0^2}{\varepsilon} - |J|\varepsilon-|J|\delta\\&
= |J|\left(\frac{M^2\phi^*(h)}{\varepsilon}-\varepsilon\right) +N\left(\frac{M^2\phi^*(h)}{\varepsilon}+\delta\right)+\frac{M^2\Theta_0^2}{\varepsilon},
    \end{aligned}
\end{equation*}

and then we get
$$\frac{1}{N}\sum\limits_{i=1}^{N}f_i(x^k)-\min\limits_{x \in Q}\frac{1}{N}\sum\limits_{i=1}^{N}f_i(x) \leq \frac{|J|}{N}\left(\frac{M^2\phi^*(h)}{\varepsilon}-\varepsilon\right) +\left(\frac{M^2\phi^*(h)}{\varepsilon}+\delta\right)+ \frac{M^2\Theta_0^2}{N\varepsilon}.$$

Recall that $\phi^*(h)=\frac{h^2M^2}{2}$, so
\begin{equation*}
    \begin{aligned}
    \frac{1}{N}\sum\limits_{i=1}^{N}f_i(x^k)-\min\limits_{x \in Q} \frac{1}{N}\sum\limits_{i=1}^{N}f_i(x) &\leq \frac{|J|}{N}\left(\frac{hM^2}{2}-\varepsilon\right) +\left(\frac{hM^2}{2}+\delta\right)+
\frac{M^2\Theta_0^2}{N\varepsilon} \\&
=\frac{|J|}{N}\left(\frac{\varepsilon}{2}-\varepsilon\right) +\left(\frac{\varepsilon}{2}+\delta\right)+
\frac{M^2\Theta_0^2}{N\varepsilon}.
    \end{aligned}
\end{equation*}

and by virtue of \eqref{accuracy_online}, we get
\begin{equation}\label{ggg}
    \frac{1}{N}\sum\limits_{i=1}^{N}f_i(x^k)-\min\limits_{x \in Q} \frac{1}{N}\sum\limits_{i=1}^{N}f_i(x) \leq \kappa.
\end{equation}

Assuming the non-negativity of the regret (i.e. the left side in \eqref{ggg}):
\begin{equation*}
    \begin{aligned}
    0\leq\sum\limits_{i=1}^{N}f_i(x^k)-f_i(x^*) &\leq|J|\left(\frac{hM^2}{2}-\varepsilon\right) +N\left(\frac{hM^2}{2}+\delta\right)+
\frac{M^2\Theta_0^2}{\varepsilon}\\&
=|J|\left(-\frac{\varepsilon}{2}\right) +N\left(\frac{\varepsilon}{2}+\delta\right)+\frac{M^2\Theta_0^2}{\varepsilon},
    \end{aligned}
\end{equation*}

so
$$
    |J|\leq N\left(1+\frac{2\delta}{\varepsilon}\right)+\frac{2M^2\Theta_0^2}{\varepsilon^2}.
$$

Suppose $\kappa\sim \varepsilon\sim \delta = \frac{C}{\sqrt{N}},$ for some $C>0$, then  we get
$$
    |J|\sim O(N) = N \left(3+\frac{2M^2\Theta_0^2}{C^2}\right).
$$
It means that the considered method is optimal for OCO, according to \cite{article:hazan_beyond_2014}.

\end{document}